\documentclass[11pt]{article}

\usepackage{amssymb, amsmath, amsthm, tikz, xspace,subfigure, graphicx}
\usepackage[left=1in,top=1in,right=1in]{geometry}
\usetikzlibrary{positioning, shapes.geometric, calc, intersections}
\tikzstyle{vertex}=[circle,fill=black,inner sep=2pt]
\tikzstyle{vertrect}=[draw,rectangle,inner sep=2pt]
\tikzstyle{vertdia}=[draw,diamond,inner sep=2pt]

\date{}

\theoremstyle{plain}
      \newtheorem{theorem}{Theorem}[section]
      
            \newtheorem{claim}[theorem]{Claim}
      \newtheorem{problem}[theorem]{Problem}
      
      \newtheorem{corollary}[theorem]{Corollary}
      
      \newtheorem{conjecture}[theorem]{Conjecture}
\theoremstyle{definition}

\theoremstyle{remark}

	\newcommand{\RR}{{\mathbb R}}

\def\twr{\mbox{\rm twr}}

\title{The Erd\H os-Szekeres problem and an induced Ramsey question}

\author{Dhruv Mubayi\thanks{Department of Mathematics, Statistics, and Computer Science, University of Illinois, Chicago, IL, 60607 USA.  Research partially supported by NSF grant DMS-1763317. Email: {\tt mubayi@uic.edu}} \and Andrew Suk\thanks{Department of Mathematics,  University of California at San Diego, La Jolla, CA, 92093 USA. Supported by an NSF CAREER award and an Alfred Sloan Fellowship. Email: {\tt asuk@ucsd.edu}.}}

\begin{document}

\maketitle

\begin{abstract}

Motivated by the Erd\H os-Szekeres convex polytope conjecture in $\RR^d$, we initiate the study of the following induced Ramsey problem for hypergraphs.  Given integers $ n > k \geq 5$, what is the minimum integer $g_k(n)$ such that any $k$-uniform hypergraph on $g_k(n)$ vertices with the property that any set of $k + 1$ vertices induces 0, 2, or 4 edges, contains an independent set of size $n$.  Our main result shows that $g_k(n) > 2^{cn^{k-4}}$, where $c = c(k)$.
\end{abstract}

\section{Introduction}

Given a finite point set $P$ in $d$-dimensional Euclidean space $\mathbb{R}^d$, we say that $P$ is in \emph{general position} if no $d+1$ members lie on a common hyperplane.  Let $ES_d(n)$ denote the minimum integer $N$, such that any set of $N$ points in $\mathbb{R}^d$ in general position contains $n$ members in \emph{convex position}, that is, $n$ points that form the vertex set of a convex polytope.  In their classic 1935 paper, Erd\H os and Szekeres~\cite{ES35} proved that in the plane, $ES_2(n) \leq 4^n$.  In 1960, they \cite{ES60} showed that $ES_2(n) \geq 2^{n-2} + 1$ and conjectured this to be sharp for every integer $n \geq 3$.  Their conjecture has been verified for $n \leq 6$~\cite{ES35,SP}, and determining the exact value of $ES_2(n)$ for $n \geq 7$ is one of the longest-standing open problems in Ramsey theory/discrete geometry.  Recently \cite{S17}, the second author asymptotically verified the Erd\H os-Szekeres conjecture by showing that $ES_2(n) = 2^{n + o(n)}$.

In higher dimensions, $d\geq 3$, much less is known about $ES_d(n)$.  In \cite{K01}, K\'arolyi showed that projections into lower-dimensional spaces can be used to bound these functions, since most generic projections preserve general position, and the preimage of a set in convex position must itself be in convex position. Hence, $ES_d(n) \leq ES_2(n)= 2^{n + o(n)}$.  However, the best known lower bound for $ES_d(n)$ is only on the order of $2^{cn^{1/(d-1)}}$, due to Ḱ\'arolyi and Valtr \cite{KV03}.  An old conjecture of F\" uredi (see Chapter 3 in \cite{mat}) says that this lower bound is essentially the truth.
\begin{conjecture}\label{rd}

For $d\geq 3$, $ES_d(n) = 2^{\Theta(n^{1/(d-1)})}$.
\end{conjecture}

It was observed by Motzkin \cite{Mot} that any set of $d + 3$ points in $\mathbb{R}^d$ in general position contains either 0, 2, or 4 $(d + 2)$-tuples not in convex position.   By defining a hypergraph $H$ whose vertices are $N$ points in $\mathbb{R}^d$ in general position, and edges are $(d+ 2)$-tuples not in convex position, then every set of $k + 1$ vertices induces 0, 2, or 4 edges.  Moreover, by Carath\'eodory's theorem (see Theorem 1.2.3 in \cite{mat}), an independent set in $H$ would correspond to a set of points in convex position.  This leads us to the following combinatorial parameter.

 Let $g_k(n)$ be the minimum integer $N$ such that any $k$-uniform hypergraph on $N$ vertices with the property that every set of $k + 1$ vertices induces 0, 2, or 4 edges, contains an independent set of size $n$.  For $k \geq 5$, the geometric construction of K\'arolyi and Valtr \cite{KV03} mentioned earlier implies that $$g_k(n) \geq ES_{k-2}(n) \geq 2^{cn^{1/(k-3)}},$$ where $c = c(k)$.  One might be tempted to prove Conjecture~\ref{rd} by establishing a similar upper bound for $g_k(n)$.  However, our main result shows that this is not possible.

\begin{theorem} \label{es} For each $n\ge k \ge 5$ there exists $c  = c(k)>0$ such that $g_k(n) > 2^{c n^{k-4}}$.
\end{theorem}

In the other direction, we can bound $g_k(n)$ from above as follows.  For $n \geq k \geq 5$ and $t\leq k$, let $h_k(t,n)$ be the minimum integer $N$ such that any $k$-uniform hypergraph on $N$ vertices with the property that any set of $k + 1$ vertices induces at most $t$ edges, contains an independent set of size $n$.  In \cite{MS18}, the authors proved the following.

\begin{theorem}[\cite{MS18}]

For $ k\geq 5$ and $t \leq k$, there is a positive constant $c' = c'(k,t)$ such that

$$h_k(t,n) \leq \twr_{t}(c'n^{k - t}\log n),$$

\noindent where $\twr$ is defined recursively as $\twr_1(x) = x$ and $\twr_{i + 1}(x)   = 2^{\twr_{i}(x)}$.
\end{theorem}

Hence, we have the following corollary.

\begin{corollary}
For $k \geq 5$, there is a constant $c' = c'(k)$ such that $$g_k(n) \leq h_k(4,n) \leq 2^{2^{2^{c'n^{k-4}\log n}}}.$$

\end{corollary}

It is an interesting open problem to improve either the upper or lower bounds for $g_k(n)$.

 \begin{problem}

 Determine the tower growth rate for $g_k(n)$.

 \end{problem}

 Actually, this Ramsey function can be generalized further as follows: for every $S \subset \{0, 1, \ldots, k\}$, define $g_k(n, S)$ to be the minimum integer $N$ such that any $N$-vertex $k$-uniform hypergraph with the property that every set of $k+1$ vertices induces $s$ edges for some $s \in S$, contains an independent set of size $n$. General results for $g_k(n,S)$ may shed light on classical Ramsey problems, but it appears difficult to determine even the tower height for any nontrivial cases.

\section{Proof of Theorem \ref{es}}
 Let $k \ge 5$ and $N=2^{c n^{k-4}}$ where $c=c_k>0$ is sufficiently small to be chosen later.
 We are to produce a $k$-uniform hypergraph $H$ on $N$ vertices with $\alpha(H) \le n$ and every $k+1$ vertices of $H$ span 0, 2, or 4 edges.
  Let $\phi:{[N] \choose k-3}\rightarrow {[k-1] \choose 2}$ be a random ${k-1 \choose 2}$-coloring,  where each color appears on each $(k-3)$-tuple independently with probability $1/{k-1 \choose 2}$.   For $f= (v_1,\ldots, v_{k-1}) \in {[N]\choose k-1}$, where $v_1<v_2<\cdots <v_{k-1}$, define the function $\chi_f: {f \choose k-3} \rightarrow {[k-1] \choose 2}$ as follows:
  for all $\{i,j\} \in {[k-1]\choose 2}$, let
$$\chi_f(f\setminus\{v_i, v_j\})=\{i,j\}.$$

 We define the $(k-1)$-uniform hypergraph $G$, whose vertex set is $[N]$, such that
$$G=G_{\phi}:=\left\{f \in {[N] \choose k-1}: \phi(f\setminus\{u,v\})=\chi_f(f\setminus\{u,v\})
 \text{ for all } \{u,v\} \in {f \choose 2}\right\}.$$
For example, if $k=4$ (which is excluded for the theorem but we allow it to illustrate this construction) then $\phi:[N] \rightarrow \{12, 13, 23\}$ and for $f = (v_1,v_2,v_3)$, where $v_1<v_2<v_3$, we have  $f \in G$ iff $\phi(v_1)=23, \phi(v_2)=13$, and $\phi(v_3)=12$.

Finally, we define  the $k$-uniform hypergraph $H$, whose vertex set is $[N]$, such that
$$H=H_{\phi}:=\left \{e \in {[N]\choose k}: |G[e]| \hbox{ is odd}\right\}.$$

\begin{claim}\label{cc1} $|H[S]|$ is even for every $S \in {[N] \choose k+1}$.\end{claim}
\begin{proof} Let $S \in {[N] \choose k+1}$ and suppose for contradiction that $|H[S]|$ is odd.
Then
$$2|G[S]| =\sum_{f \in G[S]}2 =\sum_{f \in G[S]} \sum_{e \in {S \choose k} \atop  e \supset f} 1= \sum_{e \in {S \choose k}}
|G[e]| = \sum_{e \not\in H[S]} |G[e]| +  \sum_{e \in H[S]} |G[e]|.$$
The first sum on the RHS above is even by definition of $H$ and the second sum is odd by definition of $H$ and the assumption that $|H[S]|$ is odd. This contradiction completes the proof.
\end{proof}

\begin{claim}\label{cc2} $|G[e]|\le 2$  for every $e \in {[N] \choose k}$.\end{claim}
\begin{proof}  For sake of contradiction, suppose that for $e=(v_1,\ldots, v_k)$, where $v_1<\cdots < v_k$, we have $|G[e]|\ge 3$.
Let  $e_{p}=e\setminus \{v_p\}$ for $p \in [k]$ and suppose that $e_{i}, e_{j}, e_{l} \in G$ with $i<j<l$.  In what follows, we will find a set $S$ of size $k-3$, where $S\subset e_i$ and $S\subset e_l$, such that $\chi_{e_i}(S)\neq \chi_{e_l}(S)$.  This will give us our contradiction since $e_i,e_l \in G$ implies that $\chi_{e_i}(S) = \phi(S) = \chi_{e_l}(S)$.

Let $Y=e\setminus\{v_i, v_j, v_l\}$ and  $Y'=Y\setminus \{\min Y\}$.
 Let us first assume that $i>1$ so that $\min Y=v_1$. In this case,
$$\chi_{e_i}(Y' \cup \{v_j\})=\{1,l-1\},$$
since we obtain $Y' \cup \{v_j\}$ from $e_i$ by removing $\min Y$ and $v_l$ which are the first and
$(l-1)$st elements of $e_i$. Similarly, $$\chi_{e_l}(Y' \cup \{v_j\})= \{1, i\},$$
since we obtain $Y' \cup \{v_j\}$ from $e_l$ by removing $\min Y$ and $v_i$ which are the first and
$i$th elements of $e_i$. Because $l>i+1$, we conclude that $\chi_{e_i}(Y' \cup \{v_j\})\neq \chi_{e_l}(Y' \cup \{v_j\})$ as desired.

Next, we assume that $i=1$ and $\min Y=v_q$ where $q>1$.
 In this case,
$$\chi_{e_i}(Y' \cup \{v_j\})=\{q-1,l-1\},$$
since we obtain $Y' \cup \{v_j\}$ from $e_i$ by removing $v_q$ and $v_l$ which are the $(q-1)$st and
$(l-1)$st elements of $e_i$.  Similarly, $$\chi_{e_l}(Y' \cup \{v_j\})= \{1, q'\}
\qquad \hbox{where} \qquad  q'=q \hbox{ if } q<l \hbox{ and } q'=q-1 \hbox{ if } q>l,$$
since we obtain $Y' \cup \{v_j\}$ from $e_l$ by removing $v_i=v_1$ and $v_q$ which are the first and
$q'$th elements of $e_i$. If $q\neq 2$, then we immediately obtain $\chi_{e_i}(Y' \cup \{v_j\})\neq \chi_{e_l}(Y' \cup \{v_j\})$ as desired. On the other hand, if $q=2$, then $q'=q=2$ as well and $l\ge 4$, so $l-1\neq q'$ and again
$$\chi_{e_i}(Y' \cup \{v_j\})=\{q-1,l-1\} \ne  \{1, q'\} = \chi_{e_l}(Y' \cup \{v_j\}).$$

This completes the proof of the claim.\end{proof}

Let $T_3$ be the $(k-1)$-uniform hypergraph with vertex set $S$ with $|S|=k+1$ and three edges $e_1, e_2, e_3$ such that there are three pairwise disjoint pairs $p_1, p_2, p_3 \in {S \choose 2}$  with $p_i=\{v_i, v_i'\}$ and $e_i=S\setminus p_i$ for $i
\in \{1,2,3\}$.

\begin{claim}\label{cc3} $T_3 \not\subset G$.\end{claim}

{\emph{Proof}}.  Suppose for a contradiction that there is a subset $S \subset [N]$ of size $k+1$ such that $T_3\subset G[S]$.  Using the notation above, assume without loss of generality that $v_1=\min\cup_i p_i$ and $v_2=\min (p_2 \cup p_3)$. Let $Y=S\setminus(p_1 \cup p_3)$ and note that $Y \in {e_1 \cap e_3\choose k-3}$. Let $Y_1\subset Y$ be the set of elements in $Y$ that are smaller than $v_1$, so we have the ordering
$$Y_1 < v_1 < v_2 < \{v_3, v_3'\}.$$
 Now,
$\chi_{e_1}(Y)$ is the pair of positions of $v_3$ and $v_3'$ in $e_1$. Both of these positions are at least $|Y_1|+2$ as $Y_1 \cup \{v_2\}$ lies before $p_3$. On the other hand, the smallest element of $\chi_{e_3}(Y)$ is $|Y_1|+1$ which is the position of $v_1$ in $e_3$. This shows that
$\chi_{e_1}(Y)  \ne \chi_{e_3}(Y)$, which is a contradiction as both must be equal to
$\phi(Y)$ as $e_1, e_3 \subset G$.
\qed

We now show that every $(k+1)$-set $S \subset [N]$ spans $0,2$ or $4$ edges of $H$. By Claim \ref{cc1}, $|H[S]|$ is even.  Let $G'$ be the graph with vertex set $S$ and edge set $\{S\setminus f: f \in G[S]\}$. So there is a 1-1 correspondence between $G[S]$  and $G'$ via the map $f \rightarrow S\setminus f$.  If $G'$ has a vertex $x$ of degree at least three, then $|G[S\setminus \{x\}]|\ge 3$ which contradicts Claim \ref{cc2}. Therefore $G'$ consists of disjoint paths and cycles. Next, observe that Claim \ref{cc3} implies that  $G'$ does not contain a matching of size three, for the complementary sets of this matching yield a copy of $T_3 \subset G$. This immediately implies that $k=5$, for otherwise we obtain a 3-matching in $G'$. Moreover, the only way to avoid a 3-matching when $k=5$ is for $G'$ to consist of two components each of which contains a two edge path so we may assume that $G'$ is of this form, with paths $abc, uvw$.  If both $uvw$ and $abc$ are triangles, then $|H[S]|=0$ as any 5-set $A$ in $S$ contains precisely two edges of $G'$ from $A$ to $S\setminus A$ which yields $|G[A]|=2$ so $A \notin H$. If both $abc$ and $uvw$ are paths with $deg_{G'}(b)=deg_{G'}(v)=2$, then a similar argument yields $|H[S]|=4$ (the four edges are $abcuv, abcvw, uvwab, uvwbc$) and if one is a triangle and the other is a path then we have $|H[S]|=2$. This concludes the proof that $|H[S]|\in \{0,2,4\}$ for all $S \in {[N] \choose k+1}$.

Let us now argue that $\alpha(H) \le n$, which is a straight-forward application of the probabilistic method. Indeed, we will show that this happens with positive probability and conclude that an $H$ with this property exists. For a given $k$-set, the probability that it is an edge of $H$ is $p<1$,
 where $p$ depends only on $k$. Consequently, the probability that $H$ has an independent set of size $n$ is at most
 $${N \choose n}(1-p)^{c'n^{k-3}}$$
 for some $c'>0$. Note that the exponent $k-3$ above is obtained by taking a partial Steiner $(n,k,k-3)$ system $S$ within a potential independent set of size $n$ and observing that we have independence within the edges of $S$. A short calculation shows that this probability is less than 1 as long as $c$ is sufficiently small.  This completes the proof of Theorem \ref{es} \qed

\end{document}